\documentclass[a4paper,11pt]{article}       
\usepackage{enumerate}          
\usepackage{amssymb,amsmath,amsfonts,amsthm}            
\usepackage{pdfsync}
\usepackage{url}
\usepackage{hyperref}
\usepackage{anysize}

\begin{document}
\theoremstyle{plain}
\newtheorem{definition}[equation]{Definition}
\newtheorem{lemma}{Lemma}
\newtheorem{theorem}{Theorem}
\newtheorem{metatheorem}[equation]{(Meta)-theorem}
\newtheorem{proposition}{Proposition}
\newtheorem{corollary}[equation]{Corollary}
\newtheorem{conjecture}[equation]{Conjecture}
\newtheorem{problem}{Problem}
\newtheorem{assumption}[equation]{Assumption}
\newtheorem{question}[equation]{Question}
\newtheorem{remark}{Remark}
\newtheorem{notation}[equation]{Notation}
\errorcontextlines=0

\newcommand{\R}{\mathbb{R}}
\newcommand{\e}{\varepsilon}
\renewcommand{\leq}{\leqslant}
\renewcommand{\geq}{\geqslant}

\theoremstyle{remark}
\newtheorem{note}[equation]{Note}
\newtheorem{example}[equation]{Example}

\title{From internal to pointwise control for the 1D heat equation and minimal control time}

\date{}

\author{Cyril Letrouit\footnote{Sorbonne Universit\'e, CNRS, Universit\'e de Paris, Inria, Laboratoire Jacques-Louis Lions (LJLL), F-75005 Paris, France (\texttt{letrouit@ljll.math.upmc.fr})}\ \footnote{DMA, \'Ecole normale sup\'erieure, CNRS, PSL Research University, 75005 Paris}}

\maketitle

\abstract{Our goal is to study controllability and observability properties of the 1D heat equation with internal control (or observation) set $\omega_{\varepsilon}=(x_{0}-\varepsilon, x_{0}+\varepsilon )$, in the limit $\varepsilon\rightarrow 0$, where $x_{0}\in (0,1)$. It is known that depending on arithmetic properties of $x_{0}$, there may exist a minimal time $T_{0}$ of pointwise control at $x_{0}$ of the heat equation. Besides, for any $\varepsilon$ fixed, the heat equation is controllable with control set $\omega_{\varepsilon}$ in any time $T>0$. We relate these two phenomena. We show that the observability constant on $\omega_{\e}$ does not converge to $0$ as $\e\rightarrow 0$ at the same speed when $T>T_{0}$ (in which case it is comparable to $\e^{1/2}$) or $T<T_{0}$ (in which case it converges faster to $0$). We also describe the behavior of optimal $L^{2}$ null-controls on $\omega_{\varepsilon}$ in the limit $\varepsilon \rightarrow 0$.}

\tableofcontents

\section{Introduction and main results}
\subsection{Motivations}

In this paper, we consider the controlled heat equation on $(0,1)$ with Dirichlet boundary conditions
\begin{equation} \label{heat}
\begin{cases}
\partial_{t}u-\partial_{xx} u(t,x)=f(t,x) \text{   in $(0,+\infty)\times (0,1)$} \\
u(\cdot ,0)=u(\cdot ,1)=0 \text{   on $(0,+\infty)$}, \\
u(0,\cdot)=u_{0} \text{   on (0,1)},
\end{cases}
\end{equation}
where $u_{0}(x)\in L^{2}(0,1)$ is the initial datum and $f(t,x)$ is the control. We will consider two cases in which (\ref{heat}) is known to be well-posed: 
\begin{itemize}
\item either $f\in L^{2}((0,T)\times (0,1))$;
\item or $f(t,\cdot )=\psi (t)\delta_{x_{0}}$ where $\psi\in L^{2}(0,T)$ and $x_{0}\in (0,1)$.  Here $\delta_{x_0}$ denotes the Dirac mass at $x_{0}$.
\end{itemize}

In the first case, well-posedness means that, for every $T>0$, there exists a constant $C>0$ such that for any $u_{0}\in L^{2}(0,1)$ and $f\in L^{2}((0,T)\times (0,1))$, there exists a unique solution $u\in C^{0}([0,T],L^{2}(0,1))\cap L^{2}((0,T),H_{0}^{1}(0,1))$ of (\ref{heat}), and this solution moreover satisfies 
\begin{equation*}
\| u\|_{C^{0}([0,T],L^{2}(0,1))}+\|u\|_{ L^{2}((0,T),H_{0}^{1}(0,1))} \leq C(\| u_{0}\|_{ L^{2}(0,1)} + \|f \|_{L^{2}((0,T)\times (0,1))}).
\end{equation*}
In the second case (see for example \cite[Proposition 6.1]{khodja2014minimal}), it means that, for every $T>0$, there exists a constant $C>0$ such that for any $u_{0}\in L^{2}(0,1)$ and $\psi\in L^{2}(0,T)$, there exists a unique solution $u\in C^{0}([0,T],L^{2}(0,1))\cap L^{2}((0,T),H_{0}^{1}(0,1))$ of (\ref{heat}) with $f(t,\cdot)=\psi (t)\delta_{x_{0}}$, and this solution moreover satisfies 
\begin{equation*}
\| u\|_{C^{0}([0,T],L^{2}(0,1))}+\|u\|_{ L^{2}((0,T),H_{0}^{1}(0,1))} \leq C(\| u_{0}\|_{ L^{2}(0,1)} + \|\psi \|_{L^{2}(0,T)}).
\end{equation*}

In this paper, what will be of interest is the case where $f$ is concentrated only on one point $x_{0}\in (0,1)$ (in this case we speak of \emph{pointwise} control at $x_{0}$) or on a small neighborhood of $x_{0}$ of the form $(x_{0}-\e,x_{0}+\e)$ for some small $\e>0$ (in this case we speak of \textit{internal} control). In the sequel, we fix a point $x_{0}\in (0,1)$.

Several results are known about exact observability (or, by duality, about exact controllability) of (\ref{heat}). In the sequel, by observability we always mean exact observability.

\begin{itemize}
\item By internal observability of (\ref{heat}) in time $T$ on an open  subset $E\subset (0,1)$, we mean that 
\begin{equation*}\label{obs}
C(T,E):=\inf \left\{ \int_{0}^{T}\int_{E}u(t,x)^{2}dxdt, \ \|u_{0} \|_{L^{2}(0,1)}=1, \ u \text{ solution of (\ref{heat}) with $f=0$}\right\} >0.
\end{equation*}
The constant $C(T,E)$ is called the observability constant on $E$ in time $T$.
\item By pointwise observability of (\ref{heat}) in time $T$ at a point $x_{0}\in (0,1)$, we mean that 
\begin{equation}\label{obspoint}
C(T,x_{0})=\inf \left\{ \int_{0}^{T}u(t,x_{0})^{2}dt, \ \|u_{0} \|_{L^{2}(0,1)}=1, \ u \text{ solution of (\ref{heat}) with $f=0$}\right\}>0.
\end{equation}
The constant $C(T,x_{0})$ is called the observability constant at point $x_{0}$ in time $T$.
\end{itemize}

By duality (see Lemma \ref{prop:duality}), observability in time $T$ of the heat equation on the open set $E$ is equivalent to the property that for all $u_{0}\in L^{2}(0,1)$, there exists $f\in L^{2}((0,T)\times (0,1))$ with support in $(0,T)\times E$ such that the solution $u$ of (\ref{heat}) satisfies $u(T,\cdot)=0$. In this case $f$ is called a \textit{null-control}. Similarly, pointwise observability of the heat equation at $x_{0}$ is equivalent to the property that for all $u_{0}\in L^{2}(0,1)$, there exists $\psi\in L^{2}(0,T)$ such that the solution $u$ of (\ref{heat}) with $f(t,\cdot )=\psi (t)\delta_{x_{0}}$ satisfies $u(T,\cdot)=0$.

Depending on the arithmetic properties of $x_{0}$ (mainly how well $x_{0}$ is approached by rational numbers), the heat equation may or may not be observable at point $x_{0}$ in time $T$. More precisely, we have the following result, due to \cite{dolecki1973observability} (see also \cite{khodja2014minimal}).
\begin{enumerate}
\item \textit{Given any $x_{0}\in (0,1)$, there exists $T_{0}\in [0,+\infty]$ such that if $T_0<+\infty$ and $T>T_{0}$, then the heat equation is pointwise observable at point $x_{0}$ in time $T$, and if $0<T<T_{0}\leq +\infty$, then it is not pointwise observable at point $x_{0}$ in time $T$.}
\end{enumerate} 
In the sequel, we adopt the natural convention that if $T_0=+\infty$, the inequality $T>T_0$ is never verified, even for $T=+\infty$. This means that if we write $T>T_0$, we also implicitely require that $T_0<+\infty$.

It is also known that on any open sub-interval of $(0,1)$, the heat equation is observable in any time $T>0$ (see, e.g., \cite{russell1978controllability}). In particular:
\begin{enumerate}[2.]
\item \textit{Given any $x_{0}\in (0,1)$, any $\e >0$ such that $(x_{0}-\e ,x_{0}+\e )\subset (0,1)$ and any $T>0$, the heat equation is observable on $(x_{0}-\e ,x_{0}+\e )$ in time $T$.}
\end{enumerate}
Our goal is to understand how these two phenomena are linked, most notably by studying the limit $\varepsilon \rightarrow 0$. How does a minimal time appear when the domain of observation shrinks, i.e. when $\varepsilon\rightarrow 0$? Is it related to the size of $L^{2}$-optimal null-controls in the limit $\e\rightarrow 0$? 

The appearance of a minimal time of control at $x_{0}$ can be intuitively understood in the following way. First assume that $x_{0}$ is a rational number, $x_{0}=p/q$ with $p\in \mathbb{N}$ and $q\in\mathbb{N}^{*}$. Then, for any time $T>0$, the initial datum $u_{0}=\sin (q\pi x)$ cannot be steered to $0$ by any control of the form $\psi(t)\delta_{x_{0}}$ with $\psi\in L^{2}(0,T)$. If $u$ denotes the solution of (\ref{heat}) with initial datum $u_{0}$, the quantity $\int_{0}^{T}u(t,x_{0})^{2}dt$ is equal to $0$, and therefore the heat equation is not pointwise observable at $x_{0}$. In this case, $T_0=+\infty$. If now $x_{0}$ is irrational but well approached by rational numbers, meaning that there exist sequences $(p_{k}),(n_{k})$ of integers such that $|x_{0}-p_{k}/n_{k}|$ is very small compared with $1/n_{k}$ (typically less than $e^{-Cn_{k}^{2}}$), then, by evaluating the quantity defining the observability constant (\ref{obspoint}) for the initial data $\sin (n_{k}\pi x)$, it is possible to prove that the observability constant is equal to $0$ for any $T>0$ (but the infimum defining the observability constant is not reached if $x_{0}\notin \mathbb{Q}$), so that $T_0=+\infty$.

It is also interesting to compute $T_0$ for $x_0$ an irrational algebraic number of degree $m$, that is a root of a polynomial of degree $m\geq 2$. Liouville's theorem on diophantine approximation states that in this case there exists a constant $c(x_0)$ such that $|x_0-p/n|>c(x_0)/n^m$ for all integers $p$ and $n$ where $n>0$. Therefore $|\sin(n\pi x_0)|\geq 2c(x_0)/n^{m-1}$ for any $n>0$. Hence, for any $T>0$,
\begin{equation*}
\sum_{n=1}^{\infty}\frac{\exp(-n^2\pi^2T)}{|\sin(n\pi x_0)|}\leq \frac{1}{2c(x_0)}\sum_{n=1}^{\infty}n^{m-1}\exp(-n^2\pi^2T)<+\infty,
\end{equation*} 
and according to the results of Dolecki recalled in Theorem \ref{th:dol} below, we get $T_0=0$.

In the existing literature, similar problems have been investigated. In \cite{fabre1994pointwise}, the authors study the convergence for the 1D wave equation of the $L^{2}$-optimal null-controls on a spatial interval $(x_{0}-\e,x_{0}+\e)$ and compute their blow-up rate. Our problem is somehow the same for the heat equation, but our situation is more intricate due to the appearance at the limit of a minimal control time. For the 1D heat equation, the cost of optimal controls on shrinking volume (i.e., at the limit $\e\rightarrow 0$) does not seem to have been studied. A different asymptotic question which has attracted much more attention is the cost of optimal controls in the limit $T\rightarrow 0$ for a fixed domain of observation, see \cite{laurent2018observability} for recent results in this direction. 

Let us also mention that the existence of a minimal time of control for parabolic equations has been studied a lot in the last few years. See for example \cite{ammar2011kalman} or \cite{khodja2014minimal}. However, it has apparently never been related to the blow-up of the cost of the null-controls in the limit $\e\rightarrow 0$ when the control is located in a thin domain of width $\e$, and this is precisely what we do in this paper for the 1D heat equation.

The specificity of our problem is that it is related to number theory, as already noted in \cite{dolecki1973observability}, since the main property which determines the cost of the optimal null-controls is how $x_{0}$ is approximated by rational numbers. The problem is tractable in dimension 1, but its extension to higher dimension is not easy. In some sense, the controllability at point $x_{0}$ of the heat equation is not a local problem but a global one: if the manifold $\Omega$ in which the heat equation evolves is deformed (even very far from $x_{0}$), the properties of controllability at point $x_{0}$ may change dramatically. Therefore, well-known methods such as Carleman estimates are not appropriate in this context since they are in some sense "local". To give an example, in \cite[Theorem 1.15]{laurent2018observability}, the authors have derived a lower bound on the observability constant of the heat equation in the limit $\e\rightarrow 0$ which is uniform in $x_0$, but the constants in this lower bound can probably be improved if we assume further arithmetic properties on the point $x_{0}$.

The main method we use to address this problem is the so-called moment method, which has been widely used to deal with the 1D heat equation since the seminal work \cite{fattorini1971exact}. See for example \cite{lissy2017cost} for recent results and an extensive bibliography.

The paper goes as follows. In Section \ref{SSec:mainresults} we state the main results of our paper. In Section \ref{SSec:perspectives}, we give some perspectives and open problems. Finally, in Section \ref{Sec:proofs}, we give the proofs.

\textit{Acknowledgment.} We warmly thank Emmanuel Tr\'elat for careful reading of earlier versions of this paper.

\subsection{Main results} \label{SSec:mainresults}

Our first main result is the following. It roughly says that the convergence of the internal observability constant to $0$ in the limit $\varepsilon \rightarrow 0$ is much faster when the heat equation is not pointwise observable at the limit point $x_{0}$ in time $T$ than when it is observable at $x_{0}$ in time $T$. Recall that $T_0\in [0,+\infty]$ has been defined above following the results of Dolecki \cite{dolecki1973observability}, and that it depends only on $x_0$.

\begin{theorem} \label{thrate}
Fix $x_{0}\in (0,1)$ and denote by $C(T,\e )$ the observability constant in time $T$ on the interval $(x_{0}-\e,x_{0}+\e)$. 
\begin{enumerate}
\item If $T>T_{0}$, then there exist constants $C_{1},C_{2}>0$ (depending on $T$) such that $C_{1}\e^{1/2}\leq C(T,\e) \leq C_{2}\e^{1/2}$.
\item If $T<T_{0}$, then there exist a sequence $\varepsilon_{k}\rightarrow 0$ and constants $C_{1}>0$ and $C_{2}>1/2$ (depending on $T$) such that $C(T,\varepsilon_{k})\leq C_{1}\varepsilon_{k}^{C_{2}}$.
\end{enumerate} 
\end{theorem}

\begin{remark} \label{rmk:duality} By duality, this theorem gives information on how, when $T>T_{0}$, for a fixed initial datum $u_{0}$, the norm of the $L^{2}$-optimal null-control $\psi_{\e}$ on $(x_{0}-\e,x_{0}+\e)$ behaves in the limit $\e\rightarrow 0$. It says that $\|\psi_{\e}\|_{L^{2}}$ is at most of the order of $C\e^{-1/2}$. To prove our results, we will sometimes make use of this duality between controllability and observability. We refer to Lemma \ref{prop:duality} for a precise statement on duality between controllability and observability.
\end{remark}

Our second result, which deals with the case $T\leq T_0$, gives a finer analysis of the behavior of the optimal null-control in the limit $\e\rightarrow 0$. Given an initial datum $u_{0}\in L^{2}(0,1)$ which is assumed to be not pointwise null-controllable at point $x_0$ in time $T$, we describe the behavior of the norm of the optimal control with control domain $(x_{0}-\e,x_{0}+\e)$ in the limit $\e\rightarrow 0$.

\begin{theorem} \label{initialdatum}
Let $x_0\in [0,1]$ and let $T\leq T_0$. We assume that $u_{0}\in L^{2}(0,1)$ is not pointwise null-controllable at $x_{0}$ in time $T$. Then the optimal $L^2$ null-control $\psi_{\varepsilon}$ in time $T$ with control domain $(x_{0}-\e,x_{0}+\e)$ of the heat equation with initial datum $u_0$ verifies $\e^{1/2}\|\psi_{\varepsilon}\|_{L^{2}}\rightarrow +\infty$.
\end{theorem}

\begin{remark}
Theorems \ref{thrate} and \ref{initialdatum} roughly indicate that, for a fixed initial datum $u_{0}\in L^{2}(0,1)$, the blow-up rate of the optimal null-controls $\psi_{\e}$ in the limit $\varepsilon\rightarrow 0$ determines the controllability at point $x_{0}$ and that the key quantity for measuring this rate is $\e^{1/2} \|\psi_{\e}\|_{L^{2}}$. 
\end{remark}

Our third result is a convergence result. Given a point $x_0$, an initial datum $u_0$ and assuming a uniform control of the quantity $\varepsilon^{1/2}\| \psi_\varepsilon\|_{L^2}$, where $\psi_\varepsilon$ is a null-control for $u_0$ in fixed time $T$ supported in $(x_0-\varepsilon,x_0+\varepsilon)$, we show that $\psi_\varepsilon$ converges in some sense to a pointwise null-control of $u_0$ at $x_0$ in time $T$.

\begin{theorem}\label{thprof}
Let $x_{0}\in (0,1)$ and $T>0$. Let $\delta>0$ be such that $(x_{0}-\delta, x_{0}+\delta)\subset (0,1)$ and let $u_{0}\in L^{2}(0,1)$ be an initial datum. For $0<\e <\delta$, we denote by $\psi_{\e}$ a null-control in time $T$ for $u_0$ of the heat equation with control domain $(x_{0}-\e,x_{0}+\e)$. We suppose that the quantity $\varepsilon^{1/2}\|\psi_\varepsilon\|_{L^2}$ is uniformly bounded in $\varepsilon$. Let $\varphi_{\e}(x,t)=\e\psi_{\e}\left( x_{0}+\frac{\e}{\delta}x,t \right) \in L^{2}((0,T)\times (-\delta,\delta))$. Then there exists $\varphi\in L^{2}((0,T)\times (-\delta,\delta))$ such that up to a subsequence $\varphi_{\e}\rightharpoonup \varphi$ weakly in $L^{2}((0,T)\times (-\delta,\delta))$ and $\psi (\cdot)=\frac{1}{\delta}\int_{-\delta}^{+\delta}\varphi (\cdot,x)dx\in L^{2}(0,T)$ is a pointwise null-control of $u_{0}$ at $x_{0}$ in time $T$.
\end{theorem}

 According to Theorem \ref{thrate}, this result applies for example in case $T>T_0$ and $\psi_\varepsilon$ is an optimal null-control for any $\varepsilon >0$.

\subsection{Perspectives and open questions} \label{SSec:perspectives}
In this section, we gather several conjectures and open questions related to the problem addressed in this paper.
\begin{itemize}
\item In case $T>T_{0}$, we speculate that there exists a universal constant $K$ such that we have $\varepsilon^{-1/2}C(T,\varepsilon)\rightarrow KC(T,x_{0})\in (0,+\infty )$. 
\item In the case where $T<T_{0}$, we think that there exists $C>0$ (depending on $T$) such that for all $\e >0$, we have $C(T,\e)\geq C\e^{3/2}$. This exponent is the one obtained for example if $x_{0}=p/q$ is a rational number and we evaluate the observability constant at an associate eigenfunction $\sin(q\pi x)$.  The moment method cannot work to prove this conjecture (the infinite series defining the scalar control does not converge). The only way we see to tackle it is to use Carleman estimates, like in \cite[Theorem 1.15]{laurent2018observability}.
\item It is probably true that the limit control $\frac{1}{\delta}\int_{-\delta}^{\delta}\varphi (\cdot, x)dx$ obtained in Theorem \ref{thprof} is an optimal control for $u_{0}$ from point $x_{0}$ in time $T$. Moreover, Theorem \ref{thprof} might hold without any extraction of a subsequence.
\item It is of interest to extend our results to dimension $>1$, that is to understand the behavior of the observability constant of the heat equation in a manifold $\Omega$ of dimension $>1$ when the domain of observation shrinks to a point or a submanifold. The moment method cannot work anymore in this context (it is restricted to dimension 1 since it requires the convergence of $\sum 1/\lambda_{n}$, where the $\lambda_{n}$ denote the eigenvalues of the Laplacian) and therefore Theorem \ref{thrate} cannot be easily transposed to this higher-dimensional setting, but Theorems \ref{initialdatum} and \ref{thprof} generalize well. In dimension $>1$, nodal sets play a role similar to the role of rational points in 1D and it is probable that depending on how well a measurable set $E$ is approached by nodal sets, the heat equation may or may not be exactly observable on $E$ in time $T>0$. 
\end{itemize}

\section{Proofs} \label{Sec:proofs}

Before presenting the proofs of our results, we recall the following theorem of \cite{dolecki1973observability}, which is the starting point of our analysis.
\begin{theorem}\cite[Theorem 1]{dolecki1973observability} \label{th:dol}
\begin{enumerate}[(a)]
\item If the series $\sum_{n=1}^{\infty}\frac{\exp (-n^{2}\pi^{2}T)}{|\sin (n\pi x_{0})|}$ is convergent, then the heat equation is pointwise observable at $x_{0}$ for all $T'>T$.
\item If this series is divergent, the heat equation is not pointwise observable at $x_{0}$ for $T'<T$.
\end{enumerate}
\end{theorem}
As a corollary, we get the existence of a minimal time of control (denoted by $T_{0}$) for pointwise control at point $x_{0}$, as already recalled in the introduction.

An important point to compute the blow-up rate of observability constants is to remark that the size observability constant is related to the one of the minimal control of the associated control problem. We recall the following lemma, for which we took the formulation of \cite{coron2007control} although it is much older (see \cite{rudin1991functional} for example).
\begin{lemma} \cite[Proposition 2.16]{coron2007control}\label{prop:duality}
 Let $H_{1}$ and $H_{2}$ be two Hilbert spaces. Let $\mathcal{F}$ be a linear continuous map from $H_{1}$ into $H_{2}$. Then $\mathcal{F}$ is onto if and only if there exists $c>0$ such that
\begin{equation} \label{formalobsineq} \|\mathcal{F}^{*}(x_{2})\|_{H_{1}}\geq c\|x_{2}\|_{H_{2}}, \quad \forall x_{2}\in H_{2}.\end{equation}
Moreover, if (\ref{formalobsineq}) holds for some $c>0$, there exists a linear continuous map $\mathcal{G}$ from $H_{2}$ into $H_{1}$ such that
\[\mathcal{F}\circ \mathcal{G}(x_{2})=x_{2}, \quad \forall x_{2}\in H_{2},\]
\[\|\mathcal{G}(x_{2})\|_{H_{1}}\leq \frac{1}{c}\|x_{2}\|_{H_{2}}, \quad \forall x_{2}\in H_{2}.\]
In particular, if $\mathcal{F}$ is the input-output map, this relates the observability constant with the controllability one.
\end{lemma}

\subsection{Proof of Theorem \ref{thrate}}

We successively prove Point 1 and Point 2 of Theorem \ref{thrate}. The proof of the lower bound of Point 1, which is the trickiest part, is done by modifying the Dolecki control in an $x$-independent way. For this, we essentially replace $\sin(n\pi x_0)$ by $\frac{1}{\varepsilon}\int_{x_0-\varepsilon}^{x_0+\varepsilon}\sin(n\pi x)dx$ and the key point (Lemma \ref{ineqsin}) is that, under suitable assumptions on $x_0$, $|\frac{1}{\varepsilon}\int_{x_0-\varepsilon}^{x_0+\varepsilon}\sin(n\pi x)dx|\geq c|\sin(n\pi x_0)|e^{-\delta n^2}$.

\paragraph{Point 1.}  For the upper bound, we proceed as follows. Since $x_{0}\notin \{ 0,1\}$, we know that $\sin(\pi x_0)\neq 0$ and therefore
\begin{eqnarray} C(T,\e)^{2}&\leq&\frac{2}{e^{-2\pi^{2}T}}\int_{0}^{T}\int_{x_{0}-\e}^{x_{0}+\e}e^{-2\pi^{2}t}\sin (\pi x)^{2}dxdt \nonumber\\
&\leq& \frac{e^{2\pi^{2}T}-1}{\pi^{2}}\int_{x_{0}-\e}^{x_{0}+\e}\sin (\pi x)^{2}dx\nonumber \\
& \sim & 2\e \frac{e^{2\pi^{2}T}-1}{\pi^{2}} \sin (\pi x_{0})^{2} \nonumber
\end{eqnarray}
when $\e\rightarrow 0$. Therefore, $C(T,\e)\leq C\e^{1/2}$, which proves the upper bound.

\paragraph{}

Following Remark \ref{rmk:duality} and Lemma \ref{prop:duality}, the proof of the lower bound consists roughly in proving an upper bound on the optimal null-controls $\psi_{\e}$ driving a given initial datum $u_{0}$ to $0$ in time $T$. In order to do so, we find a scalar null-control $\varphi_{\e}$ (in the sense that $\varphi_{\e}=b_{\e}(x)f_{\e}(t)$ with $\text{supp } b_{\e} \subset [x_{0}-\e,x_{0}+\e]$) which is not the optimal null-control but whose $L^{2}$ norm is of the same magnitude as the one of $\psi_{\e}$ in the limit $\e\rightarrow 0$. Said differently, for any $\e>0$ and any initial datum $u_{0}\in L^{2}(0,1)$, we find a scalar control $\varphi_{\e}$ on $[x_{0}-\e,x_{0}+\e]$ steering $u_{0}$ to $0$ and whose $L^{2}$ norm is bounded above by $C\e^{-1/2}\|u_0\|_{L^2}$ for some universal constant $C>0$ independent of $\e$ and of $u_{0}$.

As in \cite{fattorini1971exact}, for a fixed initial datum $u_{0}\in L^{2}(0,1)$ with Fourier decomposition $u_{0}(x)=\sum \mu_{n}\sin (n\pi x)$, we look for $\varphi_{\e}$ of the form $\varphi_{\e}=b_{\e}(x)f(t)$, with $b_{\e}(x)$ supported in $[x_{0}-\e,x_{0}+\e]$ and
\begin{equation*}
f(t)=\sum_{n=0}^{\infty}\frac{e^{-n^{2}\pi^{2}t}\mu_{n}}{\int_{x_{0}-\e}^{x_{0}+\e} b_{\e}(x)\sin(n\pi x)dx}\psi_{n}(t)
\end{equation*}
where $(\psi_{n})$ is a family of functions in $L^{2}(0,T)$ which is biorthogonal to the family of $L^{2}(0,T)$ functions $(e^{-n^{2}\pi^{2}t})$, meaning that  for $j,k\in\mathbb{N}$,
\begin{equation*}
\int_{0}^{T}\psi_{j}(t)e^{-k^{2}\pi^{2}t}dt = \delta_{jk}
\end{equation*}
with the Kronecker notation. 

Of course, this requires that the numbers $\int_{x_{0}-\e}^{x_{0}+\e} b_{\e}(x)\sin(n\pi x)dx$ are not too small (and in particular non-zero), so that $f\in L^{2}(0,T)$. In our construction, $b_{\e}(x)$ will be of the form $\chi_{[x_{0}-\e',x_{0}+\e']}$ for some well-chosen $\e/2\leq \e'\leq\e$, where the symbol $\chi$ denotes characteristic functions.

We now start the proof of the lower bound. It is based on several lemmas.

\begin{lemma} \label{lem:biortho}
There exists a family $(\psi_{n})_{n\in\mathbb{N}^{*}}\in L^{2}(0,T)$ biorthogonal to the family $e^{-n^{2}\pi^{2}t}$ and satisfying $\| \psi_{n}\|_{L^{2}}\leq e^{Cn}$ for every $n\in\mathbb{N}^{*}$.
\end{lemma}
\begin{proof}
This result follows for example from results of \cite{fattorini1971exact}. By \cite[estimate (3.9)]{fattorini1971exact}, we know that there exists $K>0$ such that for all $n\in \mathbb{N}$,
\begin{equation} \label{estfat1}
\| \psi_{n} \|_{L^{2}(0,T)}\leq Kn^{2} \frac{\underset{j=1}{\overset{\infty}{\prod}}\left(1+\frac{n^{2}}{j^{2}}\right)}{\displaystyle\prod_{\stackrel{j=1}{j\neq n}}^{\infty}\left(1-\frac{n^{2}}{j^{2}}\right)}.
\end{equation}

By \cite[Lemma 3.1]{fattorini1971exact}, we know that 
\begin{equation} \label{estfat2}
\frac{\underset{j=1}{\overset{\infty}{\prod}}\left(1+\frac{n^{2}}{j^{2}}\right)}{\displaystyle\prod_{\stackrel{j=1}{j\neq n}}^{\infty}\left(1-\frac{n^{2}}{j^{2}}\right)}=\exp (Cn+o(n))
\end{equation}
as $n\rightarrow +\infty$. Combining (\ref{estfat1}) and (\ref{estfat2}), we get the proof of Lemma \ref{lem:biortho}.
\end{proof}

\begin{lemma} \label{constseq}
For all $\delta >0$, there exist $C>0$ and a sequence $(\e_{j})_{j\in\mathbb{N}}$ tending to $0$ and satisfying $\e_{j}\geq\e_{j+1}\geq \e_{j}/2$ and $\phi_{n}^{j}\geq C\e_{j}e^{-n^{2}\pi^{2}\delta}$ where $\phi_{n}^{j}=\inf \{ |\e_{j}-p/n|, p\in\mathbb{Z}\}$.
\end{lemma}
\begin{proof}
We construct $(\e_{j})_{j\in\mathbb{N}}$ iteratively. First we construct $\e_{0}\in (0,1)$. 

Set $C=\left( 4\underset{n}{\sum} (n+1)e^{-n^{2}\pi^{2}\delta}\right)^{-1}$. Define also for $n\in\mathbb{N}^{*}$
\begin{equation*}
U_{0,n}=\left\{ x\in [0,1], \quad \exists p\in\mathbb{Z}, \quad\left|x-\frac{p}{n} \right|< Ce^{-n^{2}\pi^{2}\delta}\right\}
\end{equation*}
and 
\begin{equation*}
U_{0}=\bigcup_{n\in\mathbb{N}^{*}} U_{0,n}.
\end{equation*}
We search $\e_{0}\in (0,1)\backslash U_{0}$. Denoting by $|E|$ the Lebesgue measure of a set $E$, we have $|U_{0,n}|\leq 2(n+1)Ce^{-n^{2}\pi^{2}\delta}$ and therefore $|U_{0}|\leq 2C\underset{n}{\sum} (n+1)e^{-n^{2}\pi^{2}\delta}=1/2$. Hence, we can pick $\e_{0}\in (0,1)\backslash U_{0}$. 

Let us now define $\e_{j}$ (for $j\geq 0$) iteratively. Suppose that $\e_{j}$ has been defined for some $j\geq 0$. We set
\begin{equation*}
U_{j+1,n}=\left\{ x\in \left(\frac{\e_{j}}{2},\e_{j}\right), \quad\exists p\in\mathbb{Z}, \quad\left|x-\frac{p}{n} \right|< C\e_{j}e^{-n^{2}\pi^{2}\delta}\right\} \qquad \text{for $n\in\mathbb{N}^{*}$}
\end{equation*}
and 
\begin{equation*}
U_{j+1}=\bigcup_{n\in\mathbb{N}^{*}} U_{j+1,n}.
\end{equation*}
We have $|U_{j+1,n}|\leq C\e_{j}^{2}(n+1)e^{-n^{2}\pi^{2}\delta}$, and therefore $|U_{j+1}|\leq \frac{1}{4}\e_{j}^{2}\leq \frac{\e_{j}}{4}$. Hence we can pick $\e_{j+1}\in (\frac{\e_{j}}{2},\e_{j})\backslash U_{j+1}$. 

This procedure defines recursively a sequence which satisfies the statement of Lemma \ref{constseq}.
\end{proof}

\begin{lemma} \label{ineqsin}
Fix $\delta>0$. For a sequence $\e_{j}$ constructed as in Lemma \ref{constseq}, there exists $C>0$ such that 
\begin{equation} \label{eq:ineqsin}
\left| \int_{x_{0}-\e_{j}}^{x_{0}+\e_{j}}\sin (n\pi x)dx\right| \geq  C\e_{j}|\sin(n\pi x_{0})|e^{-n^{2}\pi^{2}\delta}
\end{equation}
\end{lemma}
\begin{proof}
We set $\theta_{n}=\inf \left\{ \left| x_{0}-\frac{p}{n}\right|, p \in \mathbb{Z} \right\} $ and $\phi_{n}^{j}=\inf\left\{ \left| \e_{j}-\frac{p}{n}\right|, p \in \mathbb{Z} \right\} $. We will keep these notations until the end of the proof of Theorem \ref{thrate}. Remark that $0\leq \theta_{n}\leq \frac{1}{2n}$ and  $0\leq \phi_{n}^{j}\leq \frac{1}{2n}$. In the sequel, we fix $j$ and $n$, and therefore we can write $\e_{j}=\frac{p}{n}\pm \phi_{n}^{j}$, omitting the dependence of $p$ in $j$ and $n$. There are two cases.

Let us first assume that $\e_{j}\leq \theta_{n}$. Then on $(x_{0}-\e_{j},x_{0}+\e_{j})$, the function $\sin(n\pi x)$ is of constant sign and therefore
\begin{equation*}
\left| \int_{x_{0}-\e_{j}}^{x_{0}+\e_{j}}\sin(n\pi x)dx\right|=\left| \int_{x_{0}-\e_{j}-p/n}^{x_{0}+\e_{j}-p/n}\sin(n\pi x)dx\right|\geq \left| \int_{x_{0}-\e_{j}-p/n}^{x_{0}+\e_{j}-p/n}2nxdx\right|
\end{equation*}
since $|\sin (\pi y)|\geq 2|y|$ for $|y|\leq 1/2$. Therefore
\begin{equation*}
\left| \int_{x_{0}-\e_{j}}^{x_{0}+\e_{j}}\sin(n\pi x)dx\right|\geq 4\e_{j}n\left| x_{0}-\frac{p}{n}\right|\geq  \frac{4}{\pi}\e_{j}|\sin (n\pi x_{0})|,
\end{equation*}
which proves that (\ref{eq:ineqsin}) holds in this case for $C=4/\pi$.

We now assume at the contrary that $\e_{j}\geq \theta_{n}$. We set $f(\e)=\int_{x_{0}-\e}^{x_{0}+\e}\sin (n\pi x)dx$. Then we can easily verify the following properties of $f$:
\begin{itemize}
\item If $\sin (n\pi x_{0})\geq 0$, then $f$ increases between $0$ and $1/2n$ and decreases between $1/2n$ and $1/n$. Moreover $f(0)=f(1/n)=0$.
\item If $\sin (n\pi x_{0})\leq 0$, then $f$ decreases between $0$ and $1/2n$ and increases between $1/2n$ and $1/n$. Moreover $f(0)=f(1/n)=0$.
\end{itemize}
Now we write
\begin{equation*}
\left| \int_{x_{0}-\e_{j}}^{x_{0}+\e_{j}}\sin(n\pi x)dx\right|=\left| \int_{x_{0}-\frac{p}{n}\mp \phi_{n}^{j}}^{x_{0}+\frac{p}{n}\pm \phi_{n}^{j}}\sin(n\pi x)dx\right|.
\end{equation*}
This last integral can be decomposed into three parts
\begin{equation*}
\int_{x_{0}-\frac{p}{n}\mp \phi_{n}^{j}}^{x_{0}+\frac{p}{n}\pm \phi_{n}^{j}}=\int_{x_{0}-\frac{p}{n}\mp \phi_{n}^{j}}^{x_{0}-\frac{p}{n}} + \int_{x_{0}-\frac{p}{n}}^{x_{0}+\frac{p}{n}} + \int_{x_{0}+\frac{p}{n}}^{x_{0}+\frac{p}{n}\pm \phi_{n}^{j}}.
\end{equation*}
The middle integral equals $0$ and the first one is also equal to $\int_{x_{0}+\frac{p}{n}\mp \phi_{n}^{j}}^{x_{0}+\frac{p}{n}}\sin(n\pi x)dx$. Finally we get
\begin{equation*}
\left| \int_{x_{0}-\e_{j}}^{x_{0}+\e_{j}}\sin(n\pi x)dx\right|=\left| \int_{x_{0}-\phi_{n}^{j}}^{x_{0}+\phi_{n}^{j}}\sin (n\pi x)dx\right|.
\end{equation*}
Let us finally prove that  
\begin{equation} \label{finalineq}
\left| \int_{x_{0}-\phi_{n}^{j}}^{x_{0}+\phi_{n}^{j}}\sin (n\pi x)dx\right|\geq C\phi_{n}^{j}|\sin (n\pi x_{0})|
\end{equation}
for some universal constant $C>0$. If $|\phi_{n}^{j}|\leq \theta_{n}$, as in the case $\e_{j}\leq \theta_{n}$, we easily get that (\ref{finalineq}) holds for $C=4/\pi$. If $\theta_{n}\leq \phi_{n}^{j}\leq 1/(2n)$, then we can suppose for example that $\sin (n\pi x_{0})\geq 0$. The case $\sin (n\pi x_{0})\leq 0$ can be handled similarly. The integral $\int_{x_{0}-\phi_{n}^{j}}^{x_{0}+\phi_{n}^{j}}\sin (n\pi x)dx$ is decomposed into
\begin{equation*}
\int_{x_{0}-\phi_{n}^{j}}^{x_{0}+\phi_{n}^{j}}=\int_{x_{0}-\phi_{n}^{j}}^{\frac{p}{n}}+\int_{\frac{p}{n}}^{2\frac{p}{n}-x_{0}+\phi_{n}^{j}}+\int_{2\frac{p}{n}-x_{0}+\phi_{n}^{j}}^{x_{0}+\phi_{n}^{j}}.
\end{equation*}
The first two integrals compensate and therefore
\begin{equation*}
\left|\int_{x_{0}-\phi_{n}^{j}}^{x_{0}+\phi_{n}^{j}}\sin (n\pi x)dx\right|=\left|\int_{2\frac{p}{n}-x_{0}+\phi_{n}^{j}}^{x_{0}+\phi_{n}^{j}}\sin (n\pi x)dx\right|.
\end{equation*}
The integral at the right-hand side has bounds $2\frac{p}{n}-x_{0}+\phi_{n}^{j}$ and $x_{0}+\phi_{n}^{j}$, between which $\sin (n\pi x)$ is positive. Note that for any $a$ such that $\sin (n\pi a)>0$ and any $b$ such that $\sin (n\pi x)$ is positive on $(a-b,a+b)$, we have $\int_{a-b}^{a+b}\sin(n\pi x)dx\geq b\sin (n\pi a)$. Applying this with $a=2\frac{p}{n}-x_{0}+\phi_{n}^{j}$ and $b=2(x_{0}-\frac{p}{n})$, we get
\begin{eqnarray}
\left|\int_{x_{0}-\phi_{n}^{j}}^{x_{0}+\phi_{n}^{j}}\sin (n\pi x)dx\right|&=&\left|\int_{2\frac{p}{n}-x_{0}+\phi_{n}^{j}}^{x_{0}+\phi_{n}^{j}}\sin (n\pi x)dx\right| \nonumber \\ 
&\geq& \left|\sin\left(n\pi \left(\frac{p}{n}+\phi_{n}^{j}\right)\right)\right|\left|x_{0}-\frac{p}{n}\right|  \nonumber \\ 
&\geq& 2n\phi_{n}^{j}\left|x_{0}-\frac{p}{n}\right| \text{        because $|\sin (x)|\geq \frac{2}{\pi}|x|$ for $|x|\leq \frac{\pi}{2}$}\nonumber \\
&\geq& \frac{2}{\pi}\phi_{n}^{j}|\sin(n\pi x_{0})| \nonumber \\
&\geq& C\e_{j}|\sin(n\pi x_{0})|e^{-n^{2}\pi^{2}\delta}  \nonumber
\end{eqnarray}
where we have used Lemma \ref{constseq}. This concludes the proof of Lemma \ref{ineqsin}.
\end{proof}

\begin{proof}[End of the proof of the lower bound]
We first prove that there exists $C>0$ such that for all $j\in \mathbb{N}$, we have $C\e_{j}^{1/2}\leq C(T,\e_{j})$. It will imply the lower bound of point 1 of Theorem \ref{thrate} for a particular sequence of $\e$, namely the sequence $(\e_{j})$. Fix $j\in\mathbb{N}$. Following \cite{fattorini1971exact}, we look for a control $\varphi_{\e_{j}}$ in the scalar form $\varphi_{\e_{j}}=f(t)\chi_{[-\e_{j},\e_{j}]}$ where $\chi$ denotes the characteristic function. We take 
\begin{equation*}
f(t)=\underset{n}{\sum} \frac{\mu_{n}e^{-n^{2}\pi^{2}T}}{\int_{x_{0}-\e_{j}}^{x_{0}+\e_{j}}\sin (n\pi x)dx}\psi_{n}(t).
\end{equation*}
Then 
\begin{equation*}
\| f(t)\|_{L^{2}(0,T)}\leq \underset{n}{\sum} \frac{|\mu_{n}|e^{-n^{2}\pi^{2}T}}{\left|\int_{x_{0}-\e_{j}}^{x_{0}+\e_{j}}\sin (n\pi x)dx\right|}\|\psi_{n}\|_{L^{2}(0,T)}.
\end{equation*}
Since $T>T_{0}$, by Theorem \ref{th:dol}, we can pick $\delta >0$ so that $\sum \frac{e^{-n^{2}\pi^{2}(T-2\delta)}}{|\sin(n\pi x_{0})|} <+\infty.$ This implies in particular
\begin{equation} \label{eq:convser}
\underset{n}{\sum} \frac{e^{-2n^{2}\pi^{2}(T-2\delta)}}{|\sin(n\pi x_{0})|^{2}} <+\infty.
\end{equation}
For this $\delta >0$, we take a sequence $(\varepsilon_{j})_{j\in\mathbb{N}}$ as in Lemma \ref{constseq}. We get, following Lemma \ref{ineqsin} and Lemma \ref{lem:biortho}:
\begin{equation*}
\| f(t)\|_{L^{2}(0,T)}\leq \frac{C}{\e_{j}} \underset{n}{\sum} \frac{|\mu_{n}|e^{-n^{2}\pi^{2}(T-\delta)}}{|\sin(n\pi x_{0})|}\|\psi_{n}\|_{L^{2}(0,T)}\leq \frac{C}{\e_{j}} \underset{n}{\sum} \frac{|\mu_{n}|e^{-n^{2}\pi^{2}(T-2\delta)}}{|\sin(n\pi x_{0})|}.
\end{equation*}
Using the Cauchy-Schwarz inequality, recalling that $\|u_{0}\|_{L^{2}}^{2}=\underset{n}{\sum} |\mu_{n}|^{2}$ and $\|\varphi_{\e_{j}}\|_{L^{2}((0,T)\times (0,1))}=\sqrt{2}\e_{j}^{1/2}\|f\|_{L^{2}(0,T)}$, we finally get 
\begin{equation*}
\| \psi_{\e_{j}}\|_{L^{2}}\leq \| \varphi_{\e_{j}}\|_{L^{2}}\leq \frac{C}{\e_{j}^{1/2}}\left(\underset{n}{\sum}  \frac{e^{-2n^{2}\pi^{2}(T-2\delta)}}{|\sin(n\pi x_{0})|^{2}}\right)^{1/2}\|u_{0}\|_{L^{2}}\leq \frac{C}{\e_{j}^{1/2}}\|u_{0}\|_{L^{2}}
\end{equation*}
because of (\ref{eq:convser}). By Lemma \ref{prop:duality}, we get that $C(T,\e_{j})\geq C\e_{j}^{1/2}$.

We have established the lower bound of point 1 of Theorem \ref{thrate} for the sequence $(\e_{j})$, but we must now deal with all $\e \in (-\delta ,\delta)$. We fix $\e \in (-\delta ,\delta)$ and $\e_{j}$ so that $\e/2\leq \e_{j}\leq \e$ which is possible by construction of the sequence $(\e_{j})$. Then the optimal null-control $\psi_{\e_{j}}$ is equal to $0$ outside $(x_{0}-\e_{j},x_{0}+\e_{j})$, and therefore it is also equal to $0$ outside $(x_{0}-\e,x_{0}+\e)$. We denote by $\psi_{\e}$ the optimal null-control on $(x_{0}-\e,x_{0}+\e)$. We have
\begin{equation*}
\e\| \psi_{\e}\|^{2}_{L^2}\leq 2\e_{j} \|\psi_{\e}\|^{2}_{L^2}=2\e_{j} \|\psi_{\e_{j}}\|^{2}_{L^2}\leq 2C.
\end{equation*}
Therefore the lower bound for the observability constant holds with $C$ replaced by $C/2$.
\end{proof}

\paragraph{Point 2.}
By Theorem \ref{th:dol}, since $T<T_{0}$, there exist $\delta >0$ and an increasing sequence $(n_{k})_{k\in \mathbb{N}}$ with $n_{k}\rightarrow +\infty$ when $k\rightarrow +\infty$ such that 
\begin{equation*}
|\sin (n_{k}\pi x_{0})|\leq e^{-n_{k}^{2}\pi^{2}(T+\delta)}.
\end{equation*}
Let us recall that $\theta_{n_{k}}=\inf \left\{ \left|x_{0}-\frac{p}{n_{k}}\right|, p \in \mathbb{Z}\right\}$ is the best approximation of $x_{0}$ by fractions with denominator $n_{k}$. Since $|x|\leq |\sin (\pi x)|/2$ for $x\in [-1 /2,1 /2]$, we have for a $p$ reaching the infimum in the definition of $\theta_{n_{k}}$:
\begin{equation*}
\theta_{n_{k}}=\frac{1}{n_{k}}|n_{k}x_{0}-p|\leq \frac{1}{2n_{k}}|\sin (n_{k}\pi x_{0}-p\pi)|\leq \frac{1}{2n_{k}}e^{-n_{k}^{2}\pi^{2}(T+\delta)}\leq e^{-n_{k}^{2}\pi^{2}(T+\delta)}.
\end{equation*}
Therefore, 
\begin{equation}\label{boundnk}
n_{k}^{2}\leq \frac{-\log \theta_{n_{k}}}{\pi^{2}(T+\delta)}.
\end{equation}
We set $\e_{k}=\theta_{n_{k}}$. Clearly, $\lim \e_{k} = 0$ when $k\rightarrow +\infty$ and we have 
\begin{equation*}
C(T,\e_{k})^{2}\leq \frac{e^{2n_{k}^{2}\pi^{2}T}-1}{2n_{k}^{2}\pi^{2}}\left|\int_{p/n_{k}}^{x_{0}+\theta_{n_{k}}}\sin (n_{k}\pi x)^{2}dx\right| \leq \frac{e^{2n_{k}^{2}\pi^{2}T}}{2n_{k}^{3}\pi^{2}}\int_{0}^{2n_{k}\theta_{n_{k}}}\sin (\pi y)^{2}dy.
\end{equation*}
Using that $|\sin (x)|\leq |x|$, we get 
\begin{equation*}
C(T,\e_{k})^{2}\leq \frac{e^{2n_{k}^{2}\pi^{2}T}}{2n_{k}^{3}\pi^{2}}\frac{\pi^{2}(2n_{k}\theta_{n_{k}})^{3}}{6}=\frac{2}{3}e^{2n_{k}^{2}\pi^{2}T}\theta_{n_{k}}^{3}
\end{equation*}
We can bound this expression by above using (\ref{boundnk}), and we get
\begin{equation*}
C(T,\e_{k})^{2}\leq \frac{2}{3}e^{2\pi^{2}T\left(  \frac{-\log \theta_{n_{k}}}{\pi^{2}(T+\delta)}\right)}\theta_{n_{k}}^{3} = \frac{2}{3}\e_{k}^{3-2T/(T+\delta)}.
\end{equation*}
Finally, we have
\begin{equation*}
C(T,\e_{k})^{2}\leq \sqrt{\frac{2}{3}}\e_{k}^{\frac{1}{2}+\frac{\delta}{T+\delta}}.
\end{equation*}
Setting $C_{2}=\frac{1}{2}+\frac{\delta}{T+\delta}$, we get the upper bound.

\subsection{Proof of Theorem \ref{initialdatum}}
We proceed by contradiction and assume that there exists $C>0$ and a sequence $(\e_{j})_{j\in\mathbb{N}}$, $\e_j\rightarrow 0$ such that $\e_j \|\psi_{\e_j}\|_{L^2}^{2}\leq C$. In the sequel, we omit index $j$. Let $\delta>0$ such that $(x_{0}-\delta ,x_{0}+\delta)\subset (0,1)$. For $x\in (x_{0}-\delta , x_{0}+\delta)$ and almost all $t\in (0,T)$ we set
\[ \varphi_{\e}(x,t)=\e \psi_{\e}\left( x_{0}+\frac{\e}{\delta}x,t\right)\]
with $\varphi_{\e}\in L^{2}((0,T)\times (-\delta,\delta))$.
Then for $0<\varepsilon < \delta$, we have
\[ \int_{0}^{T}\int_{-\delta}^{\delta}\varphi_{\e}(x,t)^{2}dxdt = \int_{0}^{T}\int_{-\delta}^{\delta}\e^{2}\psi_{\e}\left(x_{0}+\frac{\e}{\delta}x,t \right)^{2}dxdt=\delta\e \int_{0}^{T}\int_{x_{0}-\e}^{x_{0}+\e}\psi_{\e}(x,t)^{2}dxdt \leq C\delta.\]
Therefore, there exists $\varphi \in L^{2}((0,T)\times (-\delta,\delta))$ such that $\varphi_{\e}\rightharpoonup \varphi$ in $L^{2}((0,T)\times (-\delta,\delta))$.

For almost all $t\in (0,T)$, we set 
\[ \psi (t)=\frac{1}{\delta}\int_{-\delta}^{\delta} \varphi (x,t)dx \in L^{2}(0,T)\]
and we prove that $\psi$ is a null control from $x_{0}$ for $u_{0}$ in time $T$, i.e., the function $u$ verifying
\[ \partial_{t}u-\Delta u=\psi (t)\delta_{x_{0}}, \ \  \ \ \ \ \ u_{|t=0}=u_{0}\]
with Dirichlet boundary conditions also satisfies $u_{|t=T}=0$. In other words, $\psi (t)$, which is somehow a limit of the null-controls $\varphi_{\e}$ is also a null-control. The proof goes as follows. Fix $v_{T}\in L^{2}(0,1)$. Let $v\in L^{2}((0,1)\times (0,T))$ be a solution of the backward heat equation
\[ \partial_{t}v+\Delta v=0, \ \ \ v_{|t=T}=v_{T}\]
with Dirichlet boundary conditions. We know that for every $\e>0$, the solution $u_{\e}$ of 
\[ \partial_{t}u_{\e}-\Delta u_{\e}=\psi_{\e}, \ \ \  \ \  u_{|t=0}=u_{0}\]
with Dirichlet boundary conditions also satisfies $u_{\e|t=T}=0$, and therefore 
\[ (\partial_{t}u_{\e},v)-(\Delta u_{\e},v)=(\psi_{\e},v)\]
where the scalar product is the $L^{2}((0,1)\times (0,T))$ scalar product. Integrating by part, using the boundary conditions and the fact that $v$ is a solution of the backward heat equation, we get 
\[ (u_{\e}(\cdot , T),v(\cdot ,T))-(u_{0},v(\cdot ,0))=\int_{0}^{T}\int_{x_{0}-\e}^{x_{0}+\e}\psi_{\e}(x,t)v(x,t)dxdt\] which reduces to
\begin{equation} \label{premeq}-(u_{0},v(\cdot ,0))=\int_{0}^{T}\int_{x_{0}-\e}^{x_{0}+\e}\psi_{\e}(x,t)v(x,t)dxdt.\end{equation}
Similarly, we get
\begin{equation}\label{deuxeq} (u(\cdot , T),v(\cdot ,T))-(u_{0},v(\cdot ,0))=\int_{0}^{T}\psi (t)v(x_{0},t)dt.\end{equation}
Let us prove that $\int_{0}^{T}\int_{x_{0}-\e}^{x_{0}+\e}\psi_{\e}(x,t)v(x,t)dxdt\rightarrow \int_{0}^{T}\psi (t)v(x_{0},t)dt$ when $\varepsilon \rightarrow 0$. We have
\begin{equation}\label{troiseq}\int_{0}^{T}\int_{x_{0}-\e}^{x_{0}+\e}\psi_{\e}(x,t)v(x,t)dxdt=\frac{1}{\delta}\int_{0}^{T}\int_{-\delta}^{\delta}\varphi_{\varepsilon}(x,t)v\left(x_{0}+\frac{\e}{\delta}x,t\right)dxdt=A+B \end{equation}
where
\[ A=\frac{1}{\delta}\int_{0}^{T}\int_{-\delta}^{\delta}\varphi_{\varepsilon}(x,t)v(x_{0},t)dxdt\]
and
\[ B=\frac{1}{\delta}\int_{0}^{T}\int_{-\delta}^{\delta}\varphi_{\varepsilon}(x,t)\left( v\left(x_{0}+\frac{\e}{\delta}x,t\right)-v(x_{0},t)\right)dxdt.\]
Integrating the weak convergence  $\varphi_{\e}\rightharpoonup \varphi$, which holds in $L^{2}((-\delta,\delta)\times (0,T))$, against $\frac{1}{\delta}\mathbf{1}_{(-\delta,\delta)\times (0,T)}v(x_{0},t)$, we get 
\begin{equation}\label{ineqA}A\rightarrow \int_{0}^{T}\psi (t)v(x_{0},t)dt.\end{equation}
For $B$, we prove that $B\rightarrow 0$. The proof goes as follows. We write that 
\begin{equation*}
B^{2}\leq \left(\frac{1}{\delta}\int_{0}^{T}\int_{-\delta}^{\delta}\varphi_{\e}(x,t)^{2}dxdt\right) \left( \frac{1}{\delta}\int_{0}^{T}\int_{-\delta}^{\delta}\left| v\left(x_{0}+\frac{\e}{\delta}x,t\right)-v(x_{0},t)\right|^{2}dxdt\right)
\end{equation*}
and since the first integral is bounded above by a constant $C$, we just have to prove that the second one converges to $0$. We decompose $v$, writing $v(x,t)=\sum a_{j}\sin (j\pi x)e^{-j^{2}\pi^{2}t}$, and we get
 \begin{eqnarray}
&& \int_{0}^{T}\int_{-\delta}^{\delta}\left| v\left(x_{0}+\frac{\e}{\delta}x,t\right) -v(x_{0},t)\right|^{2}dxdt = \frac{\delta}{\e}\int_{0}^{T}\int_{-\e}^{\e}\left| v\left(x_{0}+y,t\right)-v(x_{0},t)\right|^{2}dydt \nonumber \\
&& \qquad  \qquad  \qquad  \qquad \leq \frac{2\delta \| v\|_{\infty}}{\e}\int_{0}^{T}\int_{-\e}^{\e}\left| v\left(x_{0}+y,t\right)-v(x_{0},t)\right|dydt \nonumber \\
&&\qquad  \qquad  \qquad  \qquad =\frac{2\delta \| v\|_{\infty}}{\e}\int_{0}^{T}\int_{-\e}^{\e} \sum |a_{j}|e^{-j^{2}\pi^{2}t}|\sin (j\pi (x_{0}+y))-\sin (j\pi x_{0})|dydt \nonumber \\
&&\qquad  \qquad  \qquad  \qquad \leq \frac{2\delta \| v\|_{\infty}}{\e}\int_{0}^{T}\int_{-\e}^{\e} \sum |a_{j}|j\pi ye^{-j^{2}\pi^{2}t}dydt \nonumber \\
&&\qquad  \qquad  \qquad  \qquad \leq 2\e\delta \| v\|_{\infty}\int_{0}^{T} \sum |a_{j}|j\pi e^{-j^{2}\pi^{2}t}dt \nonumber
 \end{eqnarray}
which goes to $0$ in the limit $\e\rightarrow 0$. Therefore we have obtained 
\begin{equation} \label{ineqB}
B\rightarrow 0.
\end{equation} 
Combining (\ref{premeq}), (\ref{deuxeq}), \ref{troiseq}), (\ref{ineqA}) and (\ref{ineqB}), we finally get $(u(\cdot , T),v_{T})=0.$ Since this is true for all $v_{T}$, we get that $u_{T}=0$, which means that $\psi (t)\delta_{x_{0}}$ is a null-control for $u_{0}$. This is a contradiction. It finishes the proof of Theorem \ref{initialdatum}.

\subsection{Proof of Theorem \ref{thprof}}
Theorem \ref{thprof} follows from the computations done in the proof of Theorem \ref{initialdatum}. As in the proof of Theorem \ref{initialdatum}, if we set
\[ \varphi_{\e}(x,t)=\e \psi_{\e}\left( x_{0}+\frac{\e}{\delta}x,t\right), \ \ \ \ \varphi_{\e} \in L^{2}((-\delta,\delta)\times (0,T))\]
then for $0<\varepsilon < \delta$, we have
\[ \int_{0}^{T}\int_{-\delta}^{\delta}\varphi_{\e}(x,t)^{2}dxdt \leq C\delta\]
and therefore, there exists $\varphi \in L^{2}((-\delta,\delta)\times (0,T))$ such that $\varphi_{\e}\rightharpoonup \varphi$ in $L^{2}((0,T)\times (-\delta,\delta))$.

For almost all $t\in (0,T)$, we finally set 
\[ \psi (t)=\frac{1}{\delta}\int_{-\delta}^{\delta} \varphi (x,t)dx, \ \ \ \ \psi \in L^{2}(0,T)\]
and the proof of Theorem \ref{initialdatum} shows that $\psi$ is a null-control from $x_{0}$ for $u_{0}$ in time $T$. 

As a side remark, note that if $T>T_0$, Theorem \ref{thrate} ensures that the quantity $\varepsilon^{1/2} \| \psi_\varepsilon\|_{L^2}$ is uniformly bounded in $\varepsilon$ when $\psi_\varepsilon$ is the optimal null-control steering $u_0$ to $0$ in time $T$ and with control domain $(x_0-\varepsilon,x_0+\varepsilon)$. Therefore, Theorem \ref{thprof} applies in this case.

\bibliographystyle{alpha}
\bibliography{biblioshrinking}

\end{document}